\documentclass[centertags,12pt]{amsart}
\usepackage{latexsym}
\usepackage{amsthm}
\usepackage{hyperref}
\usepackage{amssymb}
\usepackage{color}
\usepackage[utf8]{inputenc}
\usepackage{mathrsfs}
\usepackage[english]{babel}
\usepackage{setspace}
\usepackage{ textcomp }
\usepackage{graphicx}
\usepackage{float}
\usepackage{tikz}
\usepackage{soul}
\usepackage{amsmath}
\graphicspath{ {images/} }

\usepackage{comment}

\numberwithin{equation}{section}
\makeatletter
\renewcommand{\subsection}{\@startsection
	{subsection}{2}{0mm}{\baselineskip}{-0.25cm}
	{\normalfont\normalsize\bf}}
\makeatother

\newtheorem{thm}{Theorem}[section]

\newtheorem{lem}[thm]{Lemma}

\newtheorem{cor}[thm]{Corollary}

\newtheorem{remark}[thm]{Remark}

\numberwithin{equation}{section}

\newtheorem{obs}[thm]{Remark}
\theoremstyle{definition}

\allowdisplaybreaks

\newcommand{\PG}{{\rm PG}}

\newcommand{\cS}{{\mathcal S}}
\newcommand{\cC}{{\mathcal C}}

\newcommand{\cD}{\mathcal{D}}

\newcommand{\fq}{{\mathbb{F}_q}}
\newcommand{\fqs}{{\mathbb{F}_{q^2}}}

\newcommand{\cE}{\mathcal{E}}

\newcommand{\cA}{\mathcal{A}}

\newcommand{\cT}{{\mathcal T}}

\newcommand{\cQ}{{\mathcal Q}}

\makeatletter
\@ifpackageloaded{hyperref}%
{\newcommand{\mylabel}[2]
	{\protected@write\@auxout{}{\string\newlabel{#1}{{#2}{\thepage}%
				{\@currentlabelname}{\@currentHref}{}}}}}%
{\newcommand{\mylabel}[2]
	{\protected@write\@auxout{}{\string\newlabel{#1}{{#2}{\thepage}}}}}
\makeatother

\title{External points to a conic from a Baer subplane }
\begin{document}
	\author[Vincenzo Pallozzi Lavorante]{Vincenzo Pallozzi Lavorante}
	\address{Dipartimento di Matematica Pura e Applicata, Universit\'a degli Studi di Modena e Reggio Emilia}
	\email{vincenzo.pallozzilavorante@unimore.it}

	
	\begin{abstract}
		For an irreducible conic $\mathcal C$ in a Desarguesian plane of odd square order, estimating the number of points, from a Baer subplane, which are external to $\mathcal C$ is a natural problem. In this paper, a complete list of possibilities is determined for the case where $\mathcal C$ shares at least one point with the subplane.
	\end{abstract}
	
	\maketitle
	
	\begin{small}
		
		
		
	\end{small}

	\section{Introduction} 
	Combinatorial problems in finite projective planes often ask to count the number of points in the intersection of two, mostly geometrically defined, subsets. Such subsets are the subplanes, arcs, ovals, unitals, blocking-sets, their complements, and, in odd order planes, also the set of external, and that of internal points to an oval. Usually, the size of such an intersection heavily depends on the mutual position of the two subsets. Even the problem of finding non-trivial estimates on it may be challenging and any attempt to solve it would require sophisticated counting arguments.
	In a desarguesian plane, that is in a projective plane $PG(2,q)$ over a finite field $\mathbb{F}_q$, the study of this and similar combinatorial problems can greatly benefit from the powerful theory of finite fields and algebraic geometry in positive characteristic.

	A typical problem of this kind, posed in \cite{abatangelo2011mutual}, is the following. The points of a projective plane $\PG(2,q)$ fall into three classes with respect to an (absolutely) irreducible conic, namely the points lying on two tangent lines (\textit{external}), on no tangent line (\textit{internal}) and the point s on the conic. Let $\cC$ and $\cD$ be two distinct irreducible conics.  The points of $\cD$ fall into one of three subsets, namely those points $E_\cC(\cD)$ of $\cD$ that are external to $\cC$, those points $I_\cC(\cD)$ that are internal, and
	$\cC \cap \cD$. This gives rise to the functions $\varepsilon_\cC(\cD) = |E_\cC(\cD)|$ and $\iota_\cC(\cD) = |I_\cC(\cD)|$ defined over the set of all conics $\cD$ distinct from $\cC$. The combinatorial problem is to compute, or estimate the value sets of $\varepsilon= \varepsilon_\cC(\cD)$ (or equivalently of $\iota=\iota_\cC(\cD)$). A solution is given in \cite{abatangelo2011mutual}: either $\varepsilon=0,q-1,q,q+1$, or $\frac{1}{2}(q-1)-(\sqrt{q}-3)\le \varepsilon\le \frac{1}{2}(q-1)-(\sqrt{q}+3)$. 
	Let $\PG(2,q)$ be the projective plane defined over a finite field $\mathbb{F}_q$ of odd order, canonically embedded in the projective plane $\PG(2,q^2)$ over the quadratic extension $\mathbb{F}_{q^2}$ of $\mathbb{F}_q$. Let $\cC$ be an (absolutely) irreducible conic of $\PG(2,q^2)$ with homogeneous equation $F(X_0,X_1,X_2)=0$ where $F\in \mathbb{F}_{q^2}[\mathtt{X}_0,\mathtt{X}_1,\mathtt{X}_2]$ is an irreducible quadratic form. Then the number of points of $\cC$ lying in $\PG(2,q)$ is at most $4$ unless $\cC$ is a conic defined over $\mathbb{F}_q$, that is, $F\in \mathbb{F}_{q}[\mathtt{X}_0,\mathtt{X}_1,\mathtt{X}_2]$, in which case that number equals $q+1$ (this because $ 5 $ points define a unique conic and the conic would through $ 5 $ points in $ \PG(2,q) $ would be defined over $ \fq $). In this chapter we are interested in the number $E_q(\cC)$ of external points to $\cC$ which lie in $\PG(2,q)$. Since conics defined over $\mathbb{F}_{q^2}$ are not equivalent over $\mathbb{F}_q$ in general, $E_q(\cC)$ viewed as a function of $\cC$ is not expected to be constant when $\cC$ runs over all conics of $\PG(2,q^2)$. Our goal is to determine the relative value set.
	
	The main results are stated in the following theorems. 
	\begin{thm} \label{th:C1}
		Let $\cC$ be a conic in the desarguesian plane $\PG(2,q^2)$  with at least one rational point in $ \PG(2,q) $ and $q \geq 5$. Then
		\[E_q(\cC)=q^2 \mbox{ if and only if } \cC \mbox{ is defined over }\fq.\]
	\end{thm}
	
	\begin{thm} \label{th:2} In the desarguesian plane $\PG(2,q^2)$ let $\cC$ be a conic not defined over $\fq$ with at least one $\fq$-rational point. Then:
		\begin{itemize}
			\item For $q=3$, $E_q(\cC) \in \{3,4,5,6,7,8,9\}$;
			
			\item for $q=5$, $E_q(\cC) \in \{11,12,14,15,16,17,19,20,21,22,25\}$;
			\item for $q >5$, we have \[  E_q(\cC)= \frac{1}{2}(q^2+(\alpha-1)q-n_0), \] where $n_0\in \{0,1,2,3\}$ and $ \alpha \in \{1,2,3,4,5,7\}$,
			and $\alpha - n_0$ is even.
		\end{itemize}

	\end{thm} 
	
	\begin{remark}
		When $q=3$ all the values $  $ occur, that is $7$ possibilities.\\
		When $q=5$ the only values missing are $\{16,22\}$, that is $2$ out of $11$.\\
		When $q >5$ we have $13$ possibilities.
	\end{remark}
	\section{External points} Our notation and terminology are standard; see \cite{casse2006projective,hirschfeld1979projective,hirschfeld1985finite, hirschfeld1991general}. In particular, for a point $(X_0:X_1:X_2)$ of $\PG(2,q^2)$ we also use the shorter notation $X=(X_0:X_1:X_2)$.
	Let $$F(X_0,X_1,X_2)= \sum_{0\leq i,j \leq2} a_{ij}X_iX_j.$$ where $a_{ij}\in \mathbb{F}_{q^2}$, and $\det(a_{ij})\neq 0$. Then $\cC$ has equation $X^t\cA X=0$, for
	\[\cA=\begin{pmatrix}
		a_{00} & \frac{a_{01}}{2} & \frac{a_{02}}{2} \\
		\frac{a_{01}}{2} & a_{11} & \frac{a_{12}}{2} \\
		\frac{a_{02}}{2} & \frac{a_{12}}{2} & a_{22}
	\end{pmatrix}\]
	For any two distinct points $P$ and $Q$ in $\PG(2,q^2)$, the line $PQ$ meets $\cC$ in $\PG(2,q^2)$ or in a quadratic extension $\PG(2,q^4)$ of $\PG(2,q^2)$, and their common points arise from the roots $(\xi,\vartheta)$  of the homogeneous  Joachimsthal equation $$\xi^2P^t\cA P+2\xi\vartheta P^t\cA Q+\vartheta^2Q^t\cA Q =0.$$ More precisely, if  $(\xi_1,\vartheta_1)$ and $(\xi_2,\vartheta_2)$  are the (non necessarily distinct) non-$\fq$-proportional solutions of the Joachimsthal equation, then the common points are $U_i=\xi_iP+\vartheta_iQ$ for $i=1,2$. Joachimsthal equation is useful to distinguish between external and internal points of $\cC$.
	\begin{lem}[{\cite[Theorem 7.51]{casse2006projective}}]
		\label{lemA1}If $P$ runs over the set of all external points to $\cC$ then the values $P^t\cA P$ are all squares or all non-squares. For an external point $P$, if $P^t\cA P$ is a square then $Q^t\cA Q$ is a non-square for every internal point $Q$ to $\cC$.
	\end{lem}
	Therefore, in terms of the equation $X^t\cA X=\vartheta^2$ with $\vartheta \in \fqs\setminus \{ 0 \}$, the problem of determining $E_q(\cC)$ asks to find its homogeneous solutions $X=(X_0:X_1:X_2)$, with $X_i\in \mathbb{F}_q$.
	
	\section{The maximal case} \label{sec:max}
	
	We start the discussion with a conic $\cC$ defined over $\fq$. Write the equation of $\cC$ as
	\begin{equation*}\label{conicfq}
		\cC \colon aX^2+bXY+cY^2+dXZ+eYZ+fZ^2 =0,
	\end{equation*}
	with $a,b,c,d,e,f \in \fq$.

	\begin{lem} \label{L1}
		Let $\cC$ be a conic defined over $\fq$ with matrix $\cA$. For every point $P \in \PG(2,q)$ we have $P^t\cA P \in \square_{q^2}$.
	\end{lem}
	\begin{proof}
		$P^t\cA P$ is an element of $\fq$ and so a square of $\fqs$.
	\end{proof}
	
	\begin{thm}
		Let $\cC$ be a conic defined over $\fq$. The number $E_q(\cC)$ of external points to $\cC$ in $\PG(2,q^2)$ which lie in $\PG(2,q)$ is $q^2$.
	\end{thm}
	\begin{proof}
		Any irreducible conic defined over $\fq$ has $q+1$ points over $\fq$. From Lemma \ref{L1} the remaining points of $\PG(2,q)$ are either all external or all internal to the conic $\cC$. Let $t_P$ be the $\fq$-rational tangent to $\cC$ at an $\fq$-rational point $P$. Now any other point of $t_P$ defined over $\fq$ is an external point to $\cC$ (note that this set is non-empty as $t_P$ is defined over $\fq$). In particular, this means that any point of $\PG(2,q)$ is either on the conic or is external to the conic. Since $|\PG(2,q)|=q^2+q+1$ points, this implies that there are other $q^2$ external points.
	\end{proof}

	\section{Conics with at least one point in $\PG(2,q)$.}
	Up to a change of the reference system, we may assume that $\cC$ contains the point $(0:1:0)$. Then $\cC$ has equation
	\begin{equation}\label{conic1}
		\cC \colon aX^2+bXY+cXZ+dYZ+eZ^2=0
	\end{equation}
	with $a,b,c,d,e  \in \fqs$ where either $b \neq 0$ or $d \neq 0$. From now on we may assume $b \ne 0$. In case $b=0$ we can apply the collineation $(X:Y:Z)\mapsto(Z:Y:X)$ which swaps $b$ and $d$.
	\begin{lem} \label{L2}
		If $P$ runs over the set of all external points to $\cC$ then the values $P^t\cA P$ are all squares or all non-squares according as $-b c d + a d^2 + b^2e$ is a square or a non-square in $\mathbb{F}_{q^2}$.
	\end{lem}
	\begin{proof}
		Since the tangent line to $\cC$ at $Q=(0:1:0)$ has equation $bX+dZ=0$, the point $P=(-d/b:0:1)$ of $t_Q$ is external to $\cC$.
		We have $ P^T\cA P= (-bcd + ad^2 + b^2e)$. Now, Lemma \ref{L2} follows from Lemma \ref{lemA1} .
	\end{proof}
	
	\begin{obs} \label{rk:quad}
		Without loss of generality we can always suppose $-bcd + ad^2 + b^2e \in \square_{q^2}$. Indeed, if $-bcd + ad^2 + b^2e=\alpha\gamma^2$, with $\alpha \notin \square_{q^2}$, we only need to multiply by $\alpha$ the equation of $\cC$.
	\end{obs}

	Equation $ X^t\cA X = \vartheta^2 $ with $ \vartheta \in  \fqs $ can be rewritten over $ \fq $ as $ \fqs $ is a finite extension of $ \fq $, that is, the elements of $ \fqs $ are of the form $z=z_1+\epsilon z_2$ with $ z_1,z_2 \in \fq $ where $\epsilon \in \fqs$ is a root of an irreducible polynomial $  p(X) = X^2 -\omega $ over $ \fq $. Since the
	other root of $ p(X) $ is $\epsilon^q$, we have $\epsilon+\epsilon^q=0$. Thus, $X^t\cA X = \vartheta^2 $ reads over $ \fq $:
	\begin{equation} \label{sistem1con}
		\begin{cases}
			\begin{aligned}
				&a_1X^2+b_1XY+c_1XZ+d_1YZ+e_1Z^2=t_1^2+ \omega t_2^2 \\
				&a_2X^2+b_2XY+c_2XZ+d_2YZ+e_2Z^2=2t_1t_2
			\end{aligned}
		\end{cases}
	\end{equation}
	where $a=a_1+\epsilon a_2$, $b=b_1+\epsilon b_2$, $c=c_1+\epsilon c_2$, $d=d_1+\epsilon d_2$, $e=e_1+\epsilon e_2$, $\vartheta=t_1+ \epsilon t_2$ and $\omega=\epsilon^2$. Since we have $b \ne0$ or $d \ne 0$, we can assume $d_2 \ne 0$ or $b_2 \ne 0$ without loss of generality.
	From the second equation then
	\begin{equation}\label{eq:y1}
		Y=\frac{-e_2Z^2 + 2 t_1 t_2 - c_2 XZ - a_2 X^2}{d_2Z + b_2X}.
	\end{equation}
	Note that we lose the point $(0:1:0)$.
	Substituting $Y$ by the expression on the right hand side gives
	\begin{equation} \label{eq:cubic1}
		2 t_1 t_2( b_1 X + d_1Z) -(t_1^2+\omega t_2^2) (b_2 X+d_2Z )+ AX^3+BX^2Z+C XZ^2+D Z^3=0,
	\end{equation}
	where $A=-a_2 b_1 + a_1 b_2$, $B=b_2 c_1 - b_1 c_2 - a_2 d_1 + a_1 d_2$, $C=-c_2 d_1 + c_1 d_2 + b_2 e_1 - b_1 e_2$, $D=d_2 e_1 - d_1 e_2$ and $\omega=\epsilon^2$ is a non-square of $\fq$.
	Note that Equation \eqref{eq:cubic1} is equivalent to:
	\begin{equation*} \label{eq:seconda}
		(2 t_1 t_2-(a_2 X^2 + c_2 X Z + e_2 Z^2))( b_1 X + d_1Z)
		+(a_1 X^2 + c_1 X Z + e_1 Z^2-t_1^2-\omega t_2^2) (b_2 X+d_2Z )=0.
	\end{equation*}
	\begin{obs}\label{ob:EandS}
		The number of solutions $(X:Y:Z)$ of System \eqref{sistem1con} can be obtained (but it is not necessarily equal) by counting the points over $ \fq $ lying on the cubic surface $\cS \colon F(t_1,t_2,X,Z)=0$ of $\PG(3,q)$ with homogeneous equation \eqref{eq:cubic1}. Here $\PG(3,q)$ stands for the projective space over $\fq$ with homogeneous coordinates $ (t_1,t_2,X,Z) $.
	\end{obs}
	
	\begin{obs}
		Note that the conic $\cC$ of equation \eqref{conic1} is defined over $\fq$ if and only if the following hold:
		\[a_1 b_2 = a_2 b_1, \quad c_1 b_2 =c_2b_1, \quad d_1 b_2 =d_2b_1, \quad e_1 b_2 =e_2b_1. \]
	\end{obs}
	
	\begin{lem}
		With the notation above, if $ (A,B,C,D) = (0,0,0,0)$ then $\cC$ is a singular conic.
	\end{lem}
	\begin{proof}
		The determinant of the matrix associated to the polynomial \eqref{eq:cubic1} defining $\cC$ is 
		\[\frac{1}{4}(-a_1 - 
		a_2 \epsilon + (b_1 + b_2 \epsilon) (c_1 + c_2 \epsilon - (b_1 + b_2 \epsilon) (e_1 + e_2 \epsilon))),\]
		where we write every element $z$ of $\fqs$ as $z=z_1+\epsilon z_2$ with $ z_1,z_2 \in \fq $ where $\epsilon \in \fqs$ is a root of an irreducible polynomial $  p(X) = X^2 + \beta $ over $ \fq $.
		Since $d=1$ and $D=0$ we have that $e_2=0$. Furthermore, $ A=0 $, $ C=0 $ and $B=0$ imply respectively:
		\begin{equation}\label{eq:Csing}
			a_1=\frac{a_2 b_1}{b_2}, \quad  c_2 = b_2 e_1 \mbox{ and } a_2=b_2 c_1 -b_1 c_2.
		\end{equation}
		Then $\det(\cA)=0$.
	\end{proof}
	\begin{obs}
		Since by hypothesis the conic $\cC$ is non-singular, we cannot have $ (A,B,C,D) = (0,0,0,0)$.
	\end{obs}
	
	\begin{lem} \label{lm:Sirred}
		The cubic surface $\cS$, defined by the equation \emph{(\ref{eq:cubic1})}, is irreducible if and only if $b_1d_2-b_2d_1 \neq 0$.
	\end{lem}
	\begin{proof}
		Note that we can write the equation of $\cS$ as \[H(t_1,t_2,X,Z)+G(X,Z)=0,\] with $F$ of degree $1$ in $X$ and $Z$. Hence the only possibility for $\cS$ to be reducible is the following one:
		\begin{equation}\label{lem}
			(k_1 X+k_2 Z)(H_1(t_1,t_2)+G_1(X,Z))=0,
		\end{equation} where $H_1(t_1,t_2)+G_1(X,Z)$ may be reducible itself. \\
		Consider now $b_1d_2-b_2d_1 = 0$. Then the plane $\pi: b_1X+d_1Z=0 $ is a component of the cubic surface $\cS$. Indeed, using \eqref{eq:cubic1}, we have \[( b_1X +d_1 Z)(b_2 t_1^2 - 2 b_1 t_1 t_2 + b_2 \omega t_2^2 + ( b_1a_2- a_1 b_2) X^2 +( b_1 c_2-
		b_2 c_1) X Z +( b_1 e_2- b_2 e_1) Z^2 )=0.\]
		Hence $\cS$ is reducible.\\
		On the other hand, if $\cS$ is reducible, using Equation (\ref{lem}) and the identity principle of polynomials,  we have
		\begin{equation*}
			\begin{cases}
				k_1 X + k_2 Z =h(b_1 X+d_1Z)\\
				k_1 X + k_2 Z = j(b_2X+d_2Z)
			\end{cases}
		\end{equation*}
		which implies $ h(b_1X+d_1Z)=j(b_2X+d_2Z )$, for some $h,j \in \overline{\mathbb{ F}}_q^*$ and so $b_1d_2=b_2d_1 $. 
	\end{proof}
	
	\section{Irreducible case}\label{s:irred}
	For a survey on cubic surfaces see \cite{manin1986cubic}.
	In this section we suppose $\cS$ irreducible. In particular we know that $b_1d_2-b_2d_1 \neq 0$ or equivalently $\frac{d}{b} \notin \fq$.
	
	{\color{black}\begin{obs}
			In this case we can set $d=1$. Indeed, we can divide the equation \eqref{conic1} by $d$, since we have $d \ne0$ and $b \ne 0$. This also implies $b_2 \ne 0$.
	\end{obs}}
	
	We want to find a bound for the number of rational points of $\cS$. If $\cS$ is a smooth surface we have the following theorem, see \cite[Theorem 27.1 and Table 1 $ \S 31$]{manin1986cubic}.
	\begin{thm}[Weil] \label{th:Weil}
		Let $\cS$ be a smooth cubic surface defined over a finite field $\fq$. Then
		\[|\cS(\fq)|=q^2+\alpha q +1,\]
		with $\alpha \in \{-2,-1,0,1,2,3,4,5,7\}$.
	\end{thm}
	
	The missing case is when $\cS$ is singular. We start our investigation from the possible singularities of $\cS$.
	
	\begin{thm}\label{Th:singular1}
		Let $\cS$ be the cubic surface defined by equation \eqref{eq:cubic1}. Then $\cS$ has at most one singular point $P$. In this case $P$ is a double point and is defined over $\fq$.
	\end{thm}
	
	\begin{proof}
		Let $\cS \colon F(t_1,t_2,X,Z)=0$. The condition $ \frac{\partial F}{\partial t_1}=\frac{\partial F}{\partial t_2}=0 $ implies \[t_1^2-\epsilon^2t_2^2=0 \mbox{ or } X=Z=0.\]
		This means that $t_1=t_2=0$, as $\epsilon \in \fqs \setminus \fq$, or $X=Z=0$.
		When $X=Z=0$ together with  $ \frac{\partial F}{\partial X}=\frac{\partial F}{\partial Z}=0 $ imply
		\begin{equation}\label{sist:xz=0}
			\begin{cases}
				\begin{aligned}
					&2t_1 t_2 b_1 -b_2(t_1^2+\omega t_2^2)=0 \\
					&2t_1 t_2 =0
				\end{aligned}
			\end{cases}
		\end{equation}
		Hence $t_1=t_2=0$. We need to study
		\begin{equation}\label{sy:derParz}
			\begin{cases}
				\begin{aligned}
					\frac{\partial F}{\partial X}=3AX^2+2BXZ+CZ^2=0\\
					\frac{\partial F}{\partial Z}=BX^2+2CXZ+3DZ^2=0
				\end{aligned}
			\end{cases}
		\end{equation}
		If $A \neq 0$  then $Z=0$ implies $X=0$, so we can have only solutions of the form $(0:0:\beta:1)$. The system becomes
		\begin{equation}\label{sy:derParzZ0}
			\begin{cases}
				\begin{aligned}
					\frac{\partial F}{\partial X}=3AX^2+2BX+C=0\\
					\frac{\partial F}{\partial Z}=BX^2+2CX+3D=0
				\end{aligned}
			\end{cases}
		\end{equation}
		Note that System \eqref{sy:derParzZ0} has either one or two solutions (counted with multiplicity). The second case is only possible if either the two equations are proportional, namely \[3A =kB, \quad B=kC \mbox{ and } C=3kD, \] or \[B=C=D=0.\]
		In any cases we have a double root ($\frac{-1}{k}$ and $0$).  Hence we can just have one singular point, say $P=(0:0:\beta:1)$. 
		Note that this still remains true when the characteristic of the field is $3$ and $B=C=0$.	
		Furthermore, $\beta$ needs to be an element of $\fq$, otherwise $P'=(0:0:\beta^q:1)$ would be another singular point different from $ P $. \\
		Suppose now $A=0$. System \eqref{sy:derParz} becomes
		\begin{equation*}
			\begin{cases}
				\begin{aligned}
					\frac{\partial F}{\partial X}&=2BXZ+CZ^2=0\\
					\frac{\partial F}{\partial Z}&=BX^2+2CXZ+3DZ^2=0
				\end{aligned}
			\end{cases}
		\end{equation*}
		If $Z=0$ and $B \ne 0$ then we have no solutions. If $Z=0$ and $B=0$ we have only the solution $(0:0:1:0)$.
		If $Z \ne 0$ we can suppose $Z=1$:
		\begin{equation*}
			\begin{cases}
				\begin{aligned}
					\frac{\partial F}{\partial X}&=2BX+C=0\\
					\frac{\partial F}{\partial Z}&=BX^2+2CX+3D=0
				\end{aligned}
			\end{cases}
		\end{equation*}
		Note that $B$ needs to be different from $0$. Indeed, if $B=0$ then we have $C=0$ and $D=0$.
		Hence we can have at most one solution which is defined over $ \fq $.\\
		Finally, observe that $P=(0:0:X:Z)$ cannot be a triple point for $\cS$. Since both $d_2$ and $b_2$ cannot be zero, the condition $ \frac{\partial^2 F}{\partial t_1^2}(P)=\frac{\partial^2 F}{\partial t_2^2}(P)=0 $ implies $ -2(b_2 X)=-2\omega(b_2 X) $ and then $\omega=1$, which is a contradiction as $ \omega $ is a non-square of $\fq$.
	\end{proof}
	We are going to study the tangent cone at a singular point $P$ to investigate the number of points of $\cS$. See \cite{bonini2021rational}.
	Remember that the tangent cone $T_P(\cS)$ is the set of all tangent lines at a singular point $P$ of $\cS$. When $\cS$ is a cubic surface we have four possibilities for the tangent cone $T_P(\cS)$:
	\begin{itemize}
		\item a quadric cone;
		\item a line (the intersection of two planes defined over $\fqs$.);
		\item a couple of distinct planes;
		\item a repeated plane.
	\end{itemize}
	\begin{thm}
		With the notation above, the tangent cone $T_P(\cS)$ at $P=(0:0:1:0)$ or $P=(0:0:\beta:1)$ is a quadric cone  with the exception of $\beta=-1,0$. In these cases it is a couple of planes either defined over $\fqs \setminus \fq$ or over $\fq$. In particular there are $q+1$, $1$, $2q+1$ tangent lines through $P$, respectively.
	\end{thm}
	\begin{proof}
		The point $P=(0:0:1:0)$ is singular if and only if $A,B=0$ (and so $C\ne 0$ by hypothesis). In this case the associated matrix of $T_P(\cS)$ is
		\begin{equation*}
			\cT=\begin{pmatrix}
				-b_2 & b_1 & 0& 0 \\
				b_1 & -b_2 \omega& 0&0\\
				0&0&0&0\\
				0&0&0&C
			\end{pmatrix}
		\end{equation*}
		It follows that $T_P(\cS)$ is a quadric cone.\\
		When $P=(0:0:\beta:1)$ the tangent cone has the following associated matrix:
		\begin{equation*}
			\cT=\begin{pmatrix}
				-b_2\beta & b_1 \beta +1  & 0& 0 \\
				b_1\beta+1 & -b_2\beta \omega  & 0&0\\
				0&0&B+3A\beta&C+B\beta\\
				0&0&C+B\beta&3D+C\beta
			\end{pmatrix}=\begin{pmatrix}
				\cT_1 & 0 \\
				0& \cT_2
			\end{pmatrix}.
		\end{equation*}
		Note that $\beta$ satisfies $3D+C\beta=-(B\beta^2+C\beta)$ and $C+B\beta=-(3A\beta^2+B\beta)$. This implies that $|\cT|=0$. Indeed, \[|\cT_2|=-(B \beta +3A \beta^2)(B\beta+C)+(B \beta +3A \beta^2)(B\beta+C)=0.\]
		Furthermore the rank of $\cT$ is $3$ except when
		\begin{itemize}
			\item[1.] $B=C=D=0$. In this case we have $\beta=0$.
			\item[2.] $B \ne 0$ and $\beta=\frac{-B}{3A}$. This case occurs when System \eqref{sy:derParzZ0} is reduced to the single equation $X^2+2X+1=0$. Thus, $\beta=-1$.
		\end{itemize}
		In both situations the rank of $\cT$ equals $2$.
	\end{proof}
	
	We want to study the maximum number of lines through $P=(0:0:\beta:1)$ entirely contained in $\cS$. We apply to $\cS$ the invertible projectivity defined by
	\[(t_1:t_2:X:Z) \mapsto (t_1:t_2:X-\beta Z: Z),\]
	so that
	\[\cS' \colon Z (2(1+b_1 \beta) t_1 t_2 -(b_2 \beta) (t_1^2+\omega t_2^2)+(3A\beta+B)X^2)+X(2b1t_1t_2-b_2(t_1^2+\omega t_2^2)+AX^2)=0\]
	and $P'=(0:0:0:1)$.
	This means that we need to study the system
	\begin{equation}\label{sys:inter}
		\begin{cases}
			\begin{aligned}
				&\phi_2(t_1,t_2,X):=2(1+b_1 \beta) t_1 t_2 -(b_2 \beta) (t_1^2+\omega t_2^2)+(3A\beta+B)X^2=0\\
				&\phi_3(t_1,t_2,X):=X(2b_1t_1t_2-b_2(t_1^2+\omega t_2^2)+AX^2)=0
			\end{aligned}
		\end{cases}
	\end{equation}
	In fact, each point satisfying System \eqref{sys:inter} corresponds to a line through $P'$ contained in $\cS'$.
	\begin{thm}
		With the notation above, if $\alpha$ is the number of lines through the singular point $P'$, then $\alpha \in \{0,2,4\}$.
	\end{thm}
	
	\begin{proof}
		The homogeneous polynomial $\phi_2$ cannot be a factor of $\phi_3$. Thus, we have at most $6$ points of intersection. According whether $\phi_2(t_1,t_2,0)$ is irreducible or not over $\fq$ we lose or have two intersections and for every solution $(t_1:t_2:1)$ we have also the solution $(-t_1:-t_2:1)$. This implies that $\alpha \ne 1$. Now note that we have at most two solutions with $X=1$. They are given precisely by $t_1$ and $t_2$ satisfying $t_1 t_2 = c_1$, where $c_1 \in \fq \setminus \{0\}$ depends on $A, B $ and $ \beta$. 
	\end{proof}
	
	\begin{obs}
		Note that when $\phi_2$ is reducible ($\beta=0$ or $\beta=-1$) we have $\alpha=0$ when we deal with two complex planes. In particular for $\beta=0$ this cannot happen and hence $\alpha=4$.
	\end{obs}
	
	We are ready to state the main theorem.
	
	\begin{thm} \label{th:WeilSing}
		Let $\cS$ be the irreducible cubic surface defined by Equation \eqref{eq:cubic1} with a singular point, say $ P=(0:0:\beta:1) $, and let $\beta_1= (1+b_1\beta)^2-b_2^2\beta^2 \omega$. The following are the only possibilities for $\cS_q=|\cS(\fq)|$.
		\begin{equation*}
			\cS_q=\begin{cases}
				\begin{split}
					&q^2+\alpha q+1, \mbox{ if } \beta \ne 0,-1\\
					&q^2+3q+1, \mbox{ if } \beta =0\\
					&q^2+q+1, \mbox{ if } \beta_1 \notin \square_{q}, \beta=-1\\
					&q^2+(\alpha-1)q+1, \mbox{ if } \beta_1 \in \square_{q}, \beta=-1
				\end{split}
			\end{cases}
		\end{equation*}
		where $\alpha \in \{0,2,4\}$ if $\beta_1 \notin \square_{q}$ and $\alpha \in \{2,4\}$ if $\beta_1 \in \square_{q}$.
	\end{thm}
	\begin{proof}
		This proof relies on the above results. In particular, since $P$ is a double point, every line passing through $ P $, not in $T_P(\cS)$,  meets $ \cS $ in exactly one point (different from $P$). Thus, we need to subtract from $q^2+q+1$ the number of lines contained in $T_P(\cS)$ through $P$ and add $q$ whenever one of these lines lie on $\cS$.
	\end{proof}
	
	We will use the following notation.
	\begin{itemize}
		\item[] $ \cS_q $ is the number of points defined over $\PG(3,q)$ of $\cS$; moreover we set $n_0$ and $n_\infty$ to be the number of the ones with $t_1=t_2=0$ and  $X=Z=0$ respectively.
	\end{itemize}
	
	\begin{lem}
		With the above notation, let $\cC$ be a conic defined by Equation \emph{(\ref{conic1})}. Then
		\[	E_q(\cC)=\frac{1}{2}(\cS_q-n_0-n_\infty)\]
	\end{lem}
	\begin{proof}
		The points of $\cC$ can be obtained putting $\vartheta=0$ in the system (\ref{sistem1con}) whereas $E_q(\cC)$ is obtained counting the points of $\cS(\fq)$ with $\vartheta \ne 0$. This means that every point of $\cS(\fq)$ with $t_1=t_2=0$ is an $\fq$-rational point of $\cC$. Furthermore, we need to subtract the points with $X=Z=0$, since they correspond to $(0:1:0)$ which is on the conic.
		Note that for fixed $X,Y,Z$ we have either $0$ or $2$ solution for $(t_1,t_2)$ defined over $\fq$. The discriminant of the quadratic equation \eqref{eq:cubic1} in $t_1$ (or $t_2$) is actually different from $0$. This because $\omega$ is not a square in $\fq$.
		Thus, for every point  $ (X:Y:Z) $ of $\cC(\fq)$, we have two points $ (t_1:t_2:X:Z) $ of $\cS(\fq)$  so, after we subtracted the values of $n_0$ and $n_\infty$, we need to divide by two.
	\end{proof}
	
	\begin{lem}\label{lm:n0inf}
		With the above notation, we have $n_\infty=q+1$ and {\color{black} $n_0 \in \{0,1,2,3\}$}.
	\end{lem}
	\begin{proof}
		The points of $\cS(\fq)$ with $t_1=0$ and $t_2=0$ can be obtained as follows:
		\begin{itemize}
			\item $A=0$. In this case we have at least the point $(0:0:1:0)$ and at most other two points, $(0:0:\beta:1)$, with $\beta$ solution of \[BX^2+CX+D=0.\]
			\item $A \neq 0$. The points are $(0:0:\lambda:1)$, with $\lambda$ that runs over the solution set of \[AX^3+BX^2+CX+D=0,\] that are at most three.
		\end{itemize}
		The computation of $n_\infty$ follows easily. Indeed, the number of points $(t_1:t_2:0:0)$, for $t_1,t_2 \in \fq$, is $q+1$.
	\end{proof}
	Now we are ready to establish the possible values for $E_q(C)$. The values of $\cS_q$ come from Theorems \ref{th:Weil} and \ref{th:WeilSing} , for $\cS$  non-singular and singular respectively.
	\begin{cor}\label{th:sigmaq1}
		Let $\cS_q=|\cS(\fq)|$. With the notations above:
		\[ E_q(\cC)= \frac{1}{2}(\cS_q-n_0-q-1). \]
	\end{cor}

	\begin{cor}\label{th:sigmaq}
		With the notation above, the possible values for $E_q(\cC)$, when $\cS$ is non-singular are the following
		\begin{itemize}
			\item 	$ E_q(\cC)= \frac{1}{2}(q^2+(\alpha-1)q-n_0) $, with $n_0=0,2$ and $\alpha \in \{-2,0,2,4\}$
			\item 	$ E_q(\cC)= \frac{1}{2}(q^2+(\alpha-1)q-n_0) $, with $n_0=1,3$ and $\alpha \in \{-1,1,3,5,7\}$
		\end{itemize}
	\end{cor}
	\begin{proof}
		We just need to study the parity of $ q^2+(\alpha-1)q-n_0 $ to establish the possible values for $n_0$ and $\alpha$.
	\end{proof}

	\section{Reducible case}\label{s:red}
	
	Throughout this section we will suppose $b \ne 0$ and $\frac{d}{b} \in \fq$ or equivalently $b_1d_2=d_1b_2$. The case $d \ne 0$ is analogous.  \\
	Thus, from the proof of Lemma \ref{lm:Sirred} we know that $\cS$ splits as  \[\cS = \Pi \cup \cQ,\] where $\Pi$ is the plane defined by $b_1X+d_1Z=0$ (or $b_2X+d_2Z=0$) and $\cQ$ is a is a possibly degenerate quadric surface of $\PG(3,q)$ in $t_1,t_2,X,Z$.\\
	The equation of $\cS$ is
	\begin{equation*}
		(b_1 X +d_1 Z)(b_2 t_1^2 - 2 b_1 t_1 t_2 + b_2 \omega t_2^2 + ( b_1a_2- a_1 b_2) X^2 +( b_1 c_2-
		b_2 c_1) X Z +( b_1 e_2- b_2 e_1) Z^2 )
	\end{equation*}
	
	We study the two factors separately. Remember that $\cC$ is defined by equation \eqref{conic1}.
	
	\begin{lem}\label{lm:pi}
		The line of $\PG(2,q)$ defined by $b_1X + d_1Z=0$ is the tangent line to $\cC$ at the point $(0:1:0)$. In particular, it contains exactly $q$ external points to $\cC$.
	\end{lem}
	\begin{proof}
		By a straightforward computation 	\[\frac{d}{b}=\frac{d_1 b_1-\omega d_2b_2}{b_1^2-\omega b_2^2}=\frac{1}{b_2b_1}\frac{b_1^2d_2b_1-\omega b_2^2d_1b_2}{b_1^2-\omega b_2^2}=\frac{d_1b_2}{b_1b_2}=\frac{d_1}{b_1}.\] This means that the lines $bX + dZ=0$ and $b_1X + d_1Z=0$ are actually the same. In particular this is the tangent line to $\cC$ at $(0:1:0)$.
	\end{proof}
	
	\begin{obs} \label{rk:Pi}
		Lemma {\ref{lm:pi}} implies that the plane $\Pi$ is contributing with exactly $q$ solutions to System \eqref{sistem1con}, see Remark {\ref{rk:quad}}.
	\end{obs}
	
	From now on we focus on the quadric surface $\cQ$, defined by
	\begin{equation}\label{quadric}
		b_2 t_1^2 - 2 b_1 t_1 t_2 + b_2 \omega t_2^2 + ( b_1a_2- a_1 b_2) X^2 +( b_1 c_2-b_2 c_1) X Z +( b_1 e_2- b_2e_1) Z^2=0
	\end{equation}
	
	First, note that if $ b_1a_2- a_1 b_2=0 $, $ b_1 c_2-
	b_2 c_1=0 $ and $ b_1 e_2- b_2 e_1=0 $ then the conic $\cC$ is defined over $\fq$, so we can skip it now (see Section \ref{sec:max}).\\
	One associate matrix of $\cQ$ is 
	\begin{equation*}
		M=\begin{pmatrix}
			b_2 & -b_1 & 0& 0 \\
			-b_1 & b_2 \omega& 0&0\\
			0&0&a'&\frac{c'}{2}\\
			0&0&\frac{c'}{2}& e'
			
		\end{pmatrix},
	\end{equation*}
	
	where $a'= b_1a_2- a_1 b_2$, $c'= b_1 c_2-b_2 c_1$ and $e'=b_1 e_2- b_2 e_1$.
	As mentioned above, we can assume that at least one of  $a',c'$ and $ e' $ is non-zero.
	\begin{lem}\label{lm:carquad}
		Let $\delta':=c'^2-4a'e'$ and $\delta=b_2^2 \omega - b_1^2$.
		\begin{itemize}
			\item if $\delta'=0$ then $\cQ$ is a quadric cone with vertex $v=\begin{cases}
				(0:0:-c':2 a'), \mbox{ if } a' \ne 0, \\
				(0:0:1:0), \mbox{ if } a'=0
			\end{cases}.$
			\item if $\delta'\neq 0$ then $\cQ$ is a non-singular elliptic or hyperbolic quadric;
		\end{itemize}
	\end{lem}
	\begin{proof}
		The determinant of $M$ is $\Delta=-\frac{1}{4} \delta \delta'=\frac{1}{4}(b_2^2 \omega - b_1^2)(4a' e' -c'^2)$. Since $\omega$is a non-square of $\fq$ and $b \neq 0$, we have $\Delta=0$ only when $ \delta'=0 $. The proof follows from the classification of quadric surfaces. See for example \cite[pg. 14]{hirschfeld1985finite}.
	\end{proof}

	\begin{lem}\label{lm:Srid}
		We have
		\[E_q(\cC)=\frac{1}{2}\big(|\cQ|-|\cQ_0|-|\Pi \cap \cQ|+|\cQ_0 \cap \Pi|\big)+ q, \]
		where  $\cQ_0=\{(t_1,t_2,X,Z) \in \cQ | t_1=t_2=0\}$.
	\end{lem}
	\begin{proof}
		We have already seen that $\Pi$ gives its contribution of $q$ points to $E_q(\cC)$. The remaining points that contribute to $E_q(\cC)$ are those on $\cQ$ not in $\Pi$ nor $\cQ_0$ (as the points of $\cQ_0$ correspond to points of $\cC$), so the equation follows easily.
	\end{proof}
	
	From the previous lemma, we need to understand better the mutual position between $\Pi$, $\cQ$, and $\cQ_0$ to achieve our goal.

	\begin{lem}
		$\Pi$ meets $\cQ_0$  if and only if $\kappa=0$, where \[\kappa := a'd_1^2-c'd_1b_1+e'b_1^2,\]
		in which case $|\Pi \cap \cQ_0|=1$.
		Furthermore, when $\delta'=0$, $\Pi \cap \cQ_0$ is the vertex of the quadric cone $\cQ$, if $a' \ne 0$, and the empty set, if $a'=0$.
	\end{lem}
	\begin{proof}
		The plane $\Pi$ meets $\cQ_0$ only at one point, that is $(0:0:-d_1:b_1)$. Note that $\delta'=0$ and $ a'=0$ imply $c'=0$. Thus, $\kappa$ needs to be different from zero otherwise we have $e'=0$ too.
		The claim follows from standard computation.
	\end{proof}
	\begin{cor}
		We have
		\begin{equation*}
			|\cQ_0 \cap \Pi |= \begin{cases}
				\begin{aligned}
					&0, \mbox{ if } \kappa \neq 0\\
					&1,  \mbox{ if } \kappa = 0
				\end{aligned}
			\end{cases}
		\end{equation*}
	\end{cor}
	
	\begin{lem}
		\begin{equation*}
			|\cQ_0|=\begin{cases}	\begin{aligned}	&0, \mbox{ if } \delta' \notin \square_q\\&1, \mbox{ if } \delta'=0\\	&2, \mbox{ if }\delta' \in \square_q	\end{aligned}\end{cases}
		\end{equation*}
	\end{lem}
	\begin{proof}
		This result follows from standard theory. See \cite[Table 15.5]{hirschfeld1985finite} for more details.
	\end{proof}
	
	We are ready to describe the situation for every type of $\cQ$.

	\begin{thm}\label{th:Qs}
		Let $\cS=\Pi \cup \cQ$, with $\cQ$ quadric cone \emph{($\delta'=0$)}. Then
		\[E_q(\cC)=\begin{cases}
			\begin{aligned}
				\frac{1}{2}&(q^2-q)+q, \mbox{ if }\kappa =0, \delta \in \square_{q}\\
				\frac{1}{2}&(q^2+q)+q, \mbox{ if }\kappa =0, \delta \notin \square_{q}\\
				\frac{1}{2}&(q^2-1)+q, \mbox{ if } \kappa\ne 0\\
			\end{aligned}
		\end{cases}, \quad \delta=b_2^2\omega-b_1^2.\]
	\end{thm}
	\begin{proof}
		If $ \delta'=0 $,
		\begin{itemize}
			\item $\kappa=0$. This means that $\Pi$ is a plane either meeting $\cQ$ only at the vertex $v$ or through two generators of $\cQ$. More precisely, it depends on whether $\delta=b_2^2\omega-b_1^2$ is in $\square_{q}$ or not (note that $\delta$ cannot be equal to zero). Thus, we have
			\begin{equation*}
				|\Pi \cap \cQ|=\begin{cases}
					\begin{aligned}
						2q+1, \mbox{ if } \delta \in \square_{q} \\
						1, \mbox{ if } \delta \notin \square_{q}
					\end{aligned}
				\end{cases}
			\end{equation*}
			\item $\kappa\neq 0$. This implies that $\Pi$ intersects $\cQ$ in a non-singular conic with $q+1$ points, not containing $v$ and $\cQ_0=\{v\}$, where $v$ is the vertex of $\cQ$.
		\end{itemize}
		Finally, the contribution of $\Pi$ to $E_q(\cC)$ is $q$, as we have already seen.
	\end{proof}
	
	\begin{lem}\label{lm:deltano0}
		Let $\cS=\Pi \cup \cQ$. If $\delta' \ne 0$,
		\[|\Pi \cap \cQ| = \begin{cases}
			\begin{aligned}
				1&, \quad \kappa=0, \delta \notin \square_q \\
				2q+1&, \quad \kappa=0, \delta \in  \square_q \\
				q+1&, \quad \kappa\neq 0
			\end{aligned}
		\end{cases}\]
	\end{lem}
	\begin{proof}
		Note that a point $P=(t_1:t_2:X,-\frac{b_1}{d_1}X) \in \Pi \cap \cQ$ if and only if
		\[F_1(t_1,t_2)+ X^2 \kappa=0,\]
		where $F_1(t_1,t_2)=b_2 t_1^2+b_2 \omega t_2^2-2 t_1 t_2 b_1$. Thus, if $\kappa \ne 0$ there are the $q+1$ points of a non-singular conic. On the other hand, if $\kappa=0$ the number of solutions depend on whether $\delta \in \square_{q}$ or not. The claim follows by \cite[Par. 15.3]{hirschfeld1985finite}.
	\end{proof}
	
	\begin{thm}\label{th:Qns}
		Let $\cS=\Pi \cup \cQ$, with $\cQ$ a non-singular quadric surface. Then
		\begin{itemize}

			\item  if $\kappa\ne 0$ and $\delta' \in \square_q$ then \[E_q(\cC)=\begin{cases}	\begin{aligned}	&\frac{1}{2}(q^2-q-2)+q,\quad \delta \notin \square_{q},\\
					&\frac{1}{2}(q^2+q-2)+ q,\quad \delta \in \square_{q};
				\end{aligned}
			\end{cases}\]
			\item if $\kappa\ne 0$ and $\delta' \notin \square_q$ then \[E_q(\cC)=\begin{cases}
				\begin{aligned}
					&\frac{1}{2}(q^2-q)+q,\quad  \delta \in \square_{q},\\
					&\frac{1}{2}(q^2+q)+ q, \quad \delta \notin \square_{q};
				\end{aligned}
			\end{cases}\]
			\item if $\kappa=0$ \emph{(}and $\delta' \in \square_{q}$\emph{)}, then $E_q(\cC)=\frac{1}{2}(q^2-1)+ q$.
		\end{itemize}
	\end{thm}
	\begin{proof}
		The claims follows from Lemmas \ref{lm:Srid} and \ref{lm:deltano0} and from \cite[Tables 15.6 and 15.7]{hirschfeld1985finite} for the number of points of quadrics over $\fq$.
	\end{proof}
	
	\section{Proof of Theorems 1.1 and 1.2}
	We are now able to prove our main theorems. More precisely Theorem \ref{th:C1} follows from Theorems \ref{th:Weil}, \ref{th:Qs}, \ref{th:Qns} and the next result.
	\begin{thm}[{\cite{swinnerton2010cubic}}]
		In \emph{Theorem \ref{th:Weil}} the bounds are best possible, except that when $ q = 2, 3 $ or $ 5 $ the upper bound can be improved to $ \alpha \leq 5 $.
	\end{thm}
	The proof of Theorem \ref{th:2} requires one last step. For further details about the next Theorem, see \cite{ball2011introduction}.
	\begin{thm}
		Let $\cC$ be an oval of a projective plane of order $q^2$, say $\Pi_{q^2}$, with $q$ being odd. Let $\mathcal{B}$ denote a blocking set of $\Pi_{q^2}$ and $\cE$ denote the set of points lying on a tangent to $\cC$.
		
		If $\cC \cap \mathcal{B}=k$, then $|\cE \cap \mathcal{B}|\geq \frac{q^2+1-k}{2}$.
	\end{thm}
	\begin{proof}
		Each line of the plane, hence also the tangents to $\cC$, meets $\mathcal{B}$.
		If $t$ is a tangent to $\cC$ at $P\in \cC$, then $t \setminus \{P\} \subseteq \cE$ and hence $t \cap \mathcal{B} \subseteq (\cE \cap\mathcal{B})\cup \{P\}$.
		Thus $\cE \cap \mathcal{B}$ has a point of each of the tangents to the points of $\cC \setminus \mathcal{B}$.
		Since each point not in $\cC$ is incident with either $0$ or $2$ tangents to $\cC$, then $|\cE \cap \mathcal{B}| \geq \frac{q^2+1-k}{2}$.
	\end{proof}
	\begin{cor}\label{cr:rev}
		Let $\cC$ be an irreducible conic of $\PG(2,q^2)$, $ q $ odd, and let $\mathcal{B}$ denote a Baer subplane. Also, denote by $\cE$ the set of external points to $\cC$.
		
		Then $|\cE \cap \mathcal{B}|\geq \frac{q^2-3}{2}$.
	\end{cor}
	Corollary \ref{cr:rev} shows that when $q >3$ then in Theorem \ref{th:sigmaq} we can exclude the cases $\alpha=-2,-1,0$.
	Taking this into account, the proof of Theorem \ref{th:2} is a direct consequence of Theorem \ref{th:sigmaq}, \ref{th:WeilSing}, \ref{th:Qs} and \ref{th:Qns}.
	
	\section{Example}
	
	Let $\cC_1$ be a conic of equation
	\[ax^2+bxy+dyz=0,\]
	with $\frac{b}{d} \notin \fq$ and $\frac{a}{d} \in \fq$. This means that we can rewrite the equation as
	\[a'x^2+b'xy+yz=0,\]
	where $ a' \in \fq $ and $b' \in \fqs \setminus \{\fq\}$.
	Thus, we have $ 2(-bcd + ad^2 + b^2e)=a'$ which is always a square in $\fqs$.
	We conclude that the System (\ref{sistem1con}) counts the number of external points to $\cC_1$.\\
	Lemma \ref{lm:n0inf} can be refined. Indeed, we have $ a'_1 b'_2X^3=0 $ admits only the root $(0:0:0:1)$. This implies that $n_0=1$ and $C_q=2$.
	
	From Theorem \ref{th:sigmaq1}, since $\cS$ is singular with $\delta'=0$, we have just one possibility for $\sigma_q$:
	\[ \sigma_q=\frac{1}{2}(q^2+2q-1).\]
	Using the Computational Algebra System Magma \cite{Magma}, we found the following values for $(q,E_q)$: $(3,7)$, $(5,17)$, $(7,31)$, $(9,49)$.
	
	\section*{Acknowledgements}
	The research of  V. Pallozzi L. was partially supported  by the Italian National Group for Algebraic and Geometric Structures and their Applications (GNSAGA - INdAM).
	
	\bibliographystyle{IEEEtranS}
	\bibliography{bib_tesi.bib}
	
\end{document}